%% file: template.tex
\documentclass[a4paper]{amsart}
\usepackage{amssymb
	,amsthm,dirtytalk
	,amsmath,color,todonotes
	,amscd,mathrsfs
	,mathtools,enumerate
}
\usepackage{import}
\usepackage{hyperref}
\usepackage[all]{xy}
\usepackage[top=30truemm,bottom=30truemm,left=25truemm,right=25truemm]{geometry}

\newtheorem{thm}{Theorem}[section]

\newtheorem{lem}[thm]{Lemma}
\newtheorem{prop}[thm]{Proposition}
\newtheorem{coro}[thm]{Corollary}

\theoremstyle{remark}

\theoremstyle{definition}

\newtheorem{rem}[thm]{Remark}

\numberwithin{equation}{section}

\makeatletter
\def\iddots{\mathinner{\mkern1mu\raise\p@
		\hbox{.}\mkern2mu\raise4\p@\hbox{.}\mkern2mu
		\raise7\p@\vbox{\kern7\p@\hbox{.}}\mkern1mu}}
\def\adots{\mathinner{\mkern2mu\raise\p@\hbox{.}
		\mkern2mu\raise4\p@\hbox{.}\mkern1mu
		\raise7\p@\vbox{\kern7\p@\hbox{.}}\mkern1mu}}
\makeatother

\allowdisplaybreaks
\title{The WZ Pair and Finite Congruences}

\author{Xiang Chen}
\address{School of Mathematical Sciences, Beihang University, 9 Nansan Street, Shahe Higher Education Park, Changping, Beijing, 102206, P.R.C.}
\email{rinchern@buaa.edu.cn}
\author{Junyan Liu}
\email{mayljy@buaa.edu.cn}
\author{Hengfei Lu}
\email{luhengfei@buaa.edu.cn}
\date{\today}

\begin{document}
	%\maketitle
	
	\subimport{sections/}{abstract.tex}
	\maketitle
	\tableofcontents
	%\linenumbers

	\subimport{sections/}{introduction.tex}
	\subimport{sections/}{section1.tex}
	\subimport{sections/}{section2.tex}
	\subimport{sections/}{section3.tex}
\bibliographystyle{amsalpha}
\bibliography{bibs}

\end{document}

%% file: sections/abstract.tex
\begin{abstract}
We construct a meromorphic Wilf--Zeilberger pair for the hypergeometric series for \(3\ln 2\), completing the WZ proof left open by Guillera and Rogers. The same pair yields a one-parameter identity and, after a half-integer shift and differentiation, an inverse-binomial series for \(\zeta(3)\). We also study the finite truncations of the logarithmic series. For every prime \(p\ge 5\), we prove that the lower truncation is congruent to \(q_p(2)-1\) modulo \(p\), while the upper truncation is congruent to
\(-\frac{7}{16}p^2B_{p-3}\) modulo \(p^3\). The proofs combine the WZ method with terminating hypergeometric identities, the Pfaff--Saalsch\"utz identity, and classical congruences for harmonic sums.
\end{abstract}

%% file: sections/introduction.tex
\section{Introduction}

The Wilf--Zeilberger method provides an algorithmic framework for proving identities involving hypergeometric terms. Originating in the work of Wilf and Zeilberger \cite{WZ90} and developed systematically in \cite{PWZ96}, the method replaces the direct evaluation of a sum by the construction of a pair of hypergeometric functions satisfying a telescoping relation. Besides its classical applications to combinatorial and hypergeometric identities, the WZ method has proved particularly effective in the study of Ramanujan-type series and related special-value identities.

A closely related source of hypergeometric identities is the theory of logarithmic Mahler measure. Mahler measures of two-variable polynomials are known to exhibit striking connections with special values of \(L\)-functions, dilogarithms, and hypergeometric functions; see, for example, \cite{BoydRodriguezVillegas02,Rogers11}. In this direction, Guillera and Rogers \cite{GR15} used the WZ method to establish identities between Mahler measures. Among the hypergeometric evaluations arising in their work are
\[
2\ln 2
=
1+\sum_{n=1}^{\infty}
\frac{4n+1}{2n(2n+1)}
\binom{2n}{n}^{2}\frac{1}{2^{4n}},
\]
\[
3\ln 2
=
2+\sum_{n=1}^{\infty}
\frac{6n+1}{2n(2n+1)}
\binom{2n}{n}^{2}\frac{1}{2^{6n}},
\]
and
\[
8\ln 2
=
\frac{11}{2}
+
\sum_{n=1}^{\infty}
\frac{15n+2}{2n(2n+1)}
\binom{2n}{n}^{2}\frac{1}{2^{8n}}.
\]
The latter two identities are equivalent to the Mahler-measure relations
\[
4m(2)=m(8),
\qquad
11m(1)=m(16),
\]
in the notation of \cite{GR15}. Guillera and Rogers gave WZ proofs of the first and third logarithmic identities, but the corresponding WZ proof for the \(3\ln 2\) identity was left unresolved. One of the main purposes of the present paper is to complete this missing case.

\subsection{} Our first result is the construction of a meromorphic WZ pair underlying
the \(3\ln 2\) identity. Define
\[
U(n,k)
=
4^{-n}
\frac{(1/2)_n(1/2+2k)_n}{(1)_n(1+k)_n}
\frac{(1/4)_k(3/4)_k}{(1)_k^2},
\]
and
\[
F(n,k)
=
-\frac{n}{2(n+k)}U(n,k),
\qquad
G(n,k)
=
\frac{k(6n+4k+1)}
     {2(n+k)(2n+1)}U(n,k).
\]

\begin{thm}[Proposition \ref{prop:wz}, Theorem \ref{thm: parameter} and Corollary \ref{coro:3log2} ]
\label{thm:intro-wz}
The functions \(F(n,k)\) and \(G(n,k)\) form a meromorphic WZ pair:
\[
F(n+1,k)-F(n,k)
=
G(n,k+1)-G(n,k).
\]
Moreover, for every \(x>0\),
\[
\frac{4^x\sqrt{\pi}\,\Gamma(x)\Gamma(x+1)}
     {\Gamma(2x+1/2)}
=
\sum_{n=0}^{\infty}
\frac{6n+4x+1}{(2n+1)(n+x)}
\frac{(1/2)_n(1/2+2x)_n}
     {(1)_n(1+x)_n}
4^{-n}.
\]
Taking the limit \(x\to0^+\) gives
\[
3\ln 2
=
2+\sum_{n=1}^{\infty}
\frac{6n+1}{2n(2n+1)}
\binom{2n}{n}^{2}\frac{1}{2^{6n}}.
\]
\end{thm}

The same WZ pair contains further information. After a half-integer
shift in the first variable and differentiation with respect to the
second variable, it produces an inverse-binomial series for
\(\zeta(3)\).

\begin{thm}[Theorem \ref{thm:zeta3}]
\label{thm:intro-zeta3}
The following formula is true:
\[
\zeta(3)
=
\frac{4}{7}
\sum_{n=0}^{\infty}
\frac{(3n+2)4^n}
     {(2n+1)^3(n+1)\binom{2n}{n}^{2}}.
\]
\end{thm}

\subsection{}Finite truncations of hypergeometric series often possess arithmetic properties that are stronger than those suggested by termwise divisibility. This phenomenon forms part of the extensive theory of hypergeometric congruences and supercongruences. Important early examples arise from congruences for Ap\'ery-type sequences \cite{Beukers85}. Van Hamme \cite{VanHamme97} formulated a systematic collection of \(p\)-adic analogues of Ramanujan-type identities, while Rodriguez-Villegas \cite{RodriguezVillegas03} proposed further hypergeometric supercongruences motivated by Calabi--Yau geometry. Subsequent work of Mortenson \cite{Mortenson03}, Zudilin \cite{Zudilin09}, Long \cite{Long11}, Swisher \cite{Swisher15}, and Long and Ramakrishna \cite{LongRamakrishna16} developed several techniques for proving such congruences, including Gaussian hypergeometric functions, WZ pairs, classical hypergeometric transformations, and \(p\)-adic Gamma-function expansions. More recently, \(q\)-analogues and the creative \(q\)-microscope have provided another systematic approach to this subject \cite{GuoZudilin19}. Congruences involving central binomial coefficients, harmonic sums, Fermat quotients, and Bernoulli numbers have also been studied extensively; see, for instance, \cite{SunTauraso10,Sun11}.

We first use the \(2\ln 2\) identity as a model. Put
\[
a_n
=
\frac{4n+1}{2n(2n+1)}
\binom{2n}{n}^{2}\frac{1}{2^{4n}}.
\]
For an odd prime \(p\), the denominator of \(a_n\) is divisible by \(p\)
precisely at \(n=(p-1)/2\) with $1\leq n\leq p-1$. It is therefore natural to consider
separately the lower and upper parts of the truncation. Followed by the congruences in \cite{Sun14}, we can deduce that

\begin{thm}[Theorem \ref{thm:2log2 congruences}]
\label{thm:intro-2log}
Let \(p\) be an odd prime. Then
\[
\sum_{n=1}^{(p-3)/2}
\frac{4n+1}{2n(2n+1)}
\binom{2n}{n}^{2}\frac{1}{2^{4n}}
\equiv 0 \pmod p,
\]
and
\[
\sum_{n=(p+1)/2}^{p-1}
\frac{4n+1}{2n(2n+1)}
\binom{2n}{n}^{2}\frac{1}{2^{4n}}
\equiv 0 \pmod{p^3}.
\]
\end{thm}
\noindent We also examine the possibility of strengthening the upper-truncation
congruence to higher powers of \(p\). Numerical computations for primes
\(p< 300\,000\,000\) show that vanishing modulo \(p^4\) occurs only for
\(p=3\) and \(p=1237\), while neither of these two cases vanishes modulo
\(p^5\). This leads us to conjecture that the upper truncation is never
zero modulo \(p^5\) for any odd prime \(p\) (see \eqref{conj}).

For the \(3\ln 2\) identity, we likewise consider the upper and lower truncations just as we did for $2\ln 2$.

\begin{thm}[Theorem \ref{thm:3log2 congruence}]
\label{thm:intro-3log}
For every prime \(p\geq 5\),
\[
\sum_{n=1}^{(p-3)/2}
\frac{6n+1}{2n(2n+1)}
\binom{2n}{n}^{2}\frac{1}{2^{6n}}
\equiv q_p(2)-1 \pmod p,
\]
and
\begin{align*}
\sum_{n=(p+1)/2}^{p-1}
\frac{6n+1}{2n(2n+1)}
\binom{2n}{n}^{2}\frac{1}{2^{6n}}
\equiv
-\frac{7}{16}p^2B_{p-3}
\pmod{p^3},
\end{align*}
where $q_p(2)=(2^{p-1}-1)/p$ is the Fermat quotient of \(2\) and $B_{p-3}$ is the Bernoulli number.
\end{thm}
The lower-truncation congruence is obtained from a terminating specialization of the WZ pair together with the classical congruences of Morley and Lehmer \cite{Mor94,Leh38}. For the upper truncation, the WZ pair first produces a finite inverse-binomial-square identity. A differentiated Pfaff--Saalsch\"utz identity then reduces the remaining expression to harmonic sums, which can be evaluated modulo \(p\) in terms of Bernoulli polynomials and \(B_{p-3}\).

The paper is organized as follows. In Section~2 we construct the WZ pair, derive the \(3\ln 2\) identity, and obtain the corresponding formula for \(\zeta(3)\). Section~3 studies the finite congruences associated with the \(2\ln 2\) model. In Section~4 we establish the lower- and upper-truncation congruences for the \(3\ln 2\) series.

\subsection*{Acknowledgment} The third named author was partially supported by NSFC 12301031.

%% file: sections/section1.tex
\section{The Wilf--Zeilberger pair method}

The Wilf--Zeilberger method converts certain identities between infinite
series into identities between rational functions. We begin with the
formulation used by Guillera and Rogers~\cite{GR15}. We say $F(n, k)$ is hypergeometric if $F (n + 1, k)/F (n, k)$ and $F (n, k + 1)/F (n, k)$ are rational functions of $n$ and $k$. Two such functions form a WZ pair if
\begin{equation}
 F(n+1,k)-F(n,k)=G(n,k+1)-G(n,k).
 \label{eq:wz}
\end{equation}
We use meromorphic functions, with the $Pochhammer\ symbol$ interpreted as
$(a)_z=\Gamma(a+z)/\Gamma(a)$. The general WZ method is developed
in~\cite{GR15}.

Suppose that $F(0,k)=0$ and $F(n,k)\to0$ as $n\to\infty$.
Summation of \eqref{eq:wz}, we have
\[
 \sum_{n=0}^{\infty}G(n,k+1)=\sum_{n=0}^{\infty}G(n,k).
\]
If the series also converges uniformly, through countable times of induction, we have
\begin{equation}
 \sum_{n=0}^{\infty}G(n,k)
 =\sum_{n=0}^{\infty}\lim_{j\to\infty}G(n,k+j).
 \label{eq:periodic-method}
\end{equation}
If $\lim_{j\to\infty} G(n, k + j)$ is independent of $k$, then we can conclude that for unrestricted $k$:
\begin{equation}
    \sum_{n=0}^{\infty}G(n, k) =\text{constant}.
\end{equation}

For a $d$-variable Laurent polynomial $P$, the logarithmic Mahler
measure is
\[
 m(P)=\int_0^1\cdots\int_0^1
 \ln\left|P(e^{2\pi i\theta_1},\ldots,e^{2\pi i\theta_d})\right|
 \,d\theta_1\cdots d\theta_d.
\]
Following the notions in~\cite{GR15}, write
$m(\alpha)=m(\alpha+x+x^{-1}+y+y^{-1})$.
Theorem~1 of that paper gives the three identities
\begin{align}
 2\ln 2&=1+\sum_{n=1}^{\infty}
 \frac{4n+1}{2n(2n+1)}\binom{2n}{n}^{\!2}\frac1{2^{4n}},
 \label{eq:2log}\\
 3\ln 2&=2+\sum_{n=1}^{\infty}
 \frac{6n+1}{2n(2n+1)}\binom{2n}{n}^{\!2}\frac1{2^{6n}},
 \label{eq:3log}\\
 8\ln 2&=\frac{11}{2}+\sum_{n=1}^{\infty}
 \frac{15n+2}{2n(2n+1)}\binom{2n}{n}^{\!2}\frac1{2^{8n}}.
 \label{eq:8log}
\end{align}
The last two are equivalent, respectively, to
\begin{equation}
 4m(2)=m(8),\qquad 11m(1)=m(16).
 \label{eq:mahler}
\end{equation}
Guillera and Rogers gave WZ proofs for \eqref{eq:2log} and
\eqref{eq:8log}. In \cite{GR15}, they left
the WZ proof of \eqref{eq:3log} unresolved. Our first purpose is to give
a pair for this case. The same pair will subsequently
prove the associated zeta identity and the finite congruences.

\subsection{Construction of the WZ pair}
\
\newline
Let
\begin{equation}
 U(n,k)=4^{-n}\frac{(1/2)_n(1/2+2k)_n}{(1)_n(1+k)_n}
                   \frac{(1/4)_k(3/4)_k}{(1)_k^2}
 \label{eq:kernel}
\end{equation}
and
\begin{equation}
 F(n,k)=-\frac{n}{2(n+k)}U(n,k),\qquad
 G(n,k)=\frac{k(6n+4k+1)}{2(n+k)(2n+1)}U(n,k).
 \label{eq:pair}
\end{equation}

\begin{prop}
The functions in \eqref{eq:pair} form a meromorphic WZ pair.
\label{prop:wz}
\end{prop}
\begin{proof}
The shift quotients of the kernel $U(n, k)$ are
\begin{align}
 \frac{U(n+1,k)}{U(n,k)}
 &=\frac{(2n+1)(2n+4k+1)}{16(n+1)(n+k+1)},
 \label{eq:n-ratio}\\
 \frac{U(n,k+1)}{U(n,k)}
 &=\frac{(2n+4k+1)(2n+4k+3)}{16(k+1)(n+k+1)}.
 \label{eq:k-ratio}
\end{align}
Consequently,
\begin{align*}
 \frac{F(n+1,k)-F(n,k)}{U(n,k)}
 &=\frac{n}{2(n+k)}
   -\frac{(2n+1)(2n+4k+1)}{32(n+k+1)^2},\\
 \frac{G(n,k+1)-G(n,k)}{U(n,k)}
 &=\frac{(2n+4k+1)(2n+4k+3)(6n+4k+5)}
          {32(n+k+1)^2(2n+1)}\\
 &\hspace{1em}-\frac{k(6n+4k+1)}{2(n+k)(2n+1)}.
\end{align*}
The equality of these rational functions follows from
\begin{align*}
 n(2n+1)+k(6n+4k+1)&=(n+k)(2n+4k+1),\\
 (2n+1)^2+(2n+4k+3)(6n+4k+5)&=16(n+k+1)^2.
\end{align*}
This proves that $F(n,k)$ and $G(n,k)$ form a meromorphic WZ pair.
\end{proof}

\subsection{A proof of Formula \eqref{eq:3log}}
\
\newline
\indent In this section, we will prove the formula \eqref{eq:3log} about $3\ln 2$ through the WZ pair constructed in Proposition \ref{prop:wz}.
\begin{thm}\label{thm: parameter}
  The following identity is true:
  \begin{equation}\label{eq:parameter}
 \frac{4^x\sqrt\pi\,\Gamma(x)\Gamma(x+1)}{\Gamma(2x+1/2)}
 =\sum_{n=0}^{\infty}
 \frac{6n+4x+1}{(2n+1)(n+x)}
 \frac{(1/2)_n(1/2+2x)_n}{(1)_n(1+x)_n}\,4^{-n}.
 \end{equation}
 for $x\in \mathbb{R}_{>0}$.
\end{thm}
\begin{proof}
We have
\begin{align*}
 &U(n,k)=4^{-n}\frac{(1/2)_n(1/2+2k)_n}{(1)_n(1+k)_n}
                   \frac{(1/4)_k(3/4)_k}{(1)_k^2}\\
F(n,k)&=-\frac{n}{2(n+k)}U(n,k),\
 G(n,k)=\frac{k(6n+4k+1)}{2(n+k)(2n+1)}U(n,k).         
 \end{align*}
as a WZ pair and $F(0,k)=0$ is trivial. First, we will prove $F(n,k)\rightarrow 0$ as $n\rightarrow \infty$.

For $k>0$, we have
\[
0<\frac{(1/2+2k)_n}{(1+k)_n}=\prod_{j=0}^{n-1}\frac{2k+j+1/2}{k+j+1}\leq 2^n.
\]
Hence
\begin{equation}\label{eq:Unk estimate}
  0<U(n,k)\leq \dfrac{(1/4)_k(3/4)_k}{(1)_k^2}\dfrac{(1/2)_n}{(1)_n}2^{-n}=\dfrac{(1/4)_k(3/4)_k}{(1)_k^2}\binom{2n}{n}8^{-n}.
\end{equation}
Then
\[
\lim_{n\to\infty}F(n,k)=0.
\]

Summing the WZ relation over $n\ge 0$, we obtain
\[
 \sum_{n=0}^{\infty}G(n,k+1)=\sum_{n=0}^{\infty}G(n,k).
\]
To determine its value, fix $k_0>0$. By the Gamma-ratio asymptotic $\frac{\Gamma(z+a)}{\Gamma(z+b)}\sim z^{a-b}$ as $z\rightarrow\infty$ (see \cite[Formula 5.11.12]{DLMF}),
\begin{equation}\label{eq:C(k)estimate}
  k\frac{(1/4)_k(3/4)_k}{(1)_k^2}=k\frac{1}{\Gamma(1/4)\Gamma(3/4)}\frac{\Gamma(k+1/4)}{\Gamma(k+1)}\frac{\Gamma(k+3/4)}{\Gamma(k+1)}\longrightarrow
\frac{1}{\Gamma(1/4)\Gamma(3/4)}
=
\frac{1}{\pi\sqrt{2}}
\ (k\to\infty),
\end{equation}
and so it is bounded for $k\ge k_0$. Moreover,
\begin{equation}\label{eq:6n/n estimate}
  \frac{6n+4k+1}{n+k}\le 6+\frac{1}{k_0}.
\end{equation}
Followed by \eqref{eq:Unk estimate},\eqref{eq:C(k)estimate} and \eqref{eq:6n/n estimate}, we have
\[
0\le G(n,k)\leq C_{k_0}
\dfrac{\binom{2n}{n}8^{-n}}{2n+1},\quad\mbox{where  } C_{k_0}\ \mathrm{is\ a\ constant\ only\  about}\ k_0
\]
and the majorant on the right is summable. Hence the dominated convergence applies.

For fixed $n$, as $k\rightarrow\infty$,
\[
\frac{(1/2+2k)_n}{(1+k)_n}\longrightarrow 2^n,\ k\dfrac{(1/4)_k(3/4)_k}{(1)_k^2}\longrightarrow \dfrac{1}{\pi\sqrt{2}},\ \frac{6n+4k+1}{n+k}\longrightarrow 4
\]
and consequently
\[
G(n,k)\longrightarrow
\frac{\sqrt{2}}{\pi}
\dfrac{\binom{2n}{n}8^{-n}}{2n+1}.
\]
Since $\sum_{n=0}^{\infty}G(n,k+1)=\sum_{n=0}^{\infty}G(n,k)$ for every integer $j\ge 0$, letting
$j\to\infty$ gives
\[
\sum_{n=0}^{\infty}G(n,k)=
\frac{\sqrt{2}}{\pi}
\sum_{n=0}^{\infty}
\dfrac{\binom{2n}{n}8^{-n}}{2n+1}.
\]
Using
\[
\sum_{n=0}^{\infty}\frac{\binom{2n}{n}4^{-n}t^{2n}}{2n+1}
=
\frac{\arcsin t}{t},
\]
and taking $t=1/\sqrt{2}$, we obtain
\[
\sum_{n=0}^{\infty}G(n,k)=\frac12.
\]

On the other hand, from the definition of $G(n,k)$,
\[
\sum_{n=0}^{\infty}G(n,k)
=
\frac{k(1/4)_k(3/4)_k}{2(1)_k^2}
\sum_{n=0}^{\infty}
\frac{6n+4k+1}{(2n+1)(n+k)}
\frac{(1/2)_n(1/2+2k)_n}
{(1)_n(1+k)_n}
4^{-n}.
\]
Therefore,
\[
\sum_{n=0}^{\infty}
\frac{6n+4k+1}{(2n+1)(n+k)}
\frac{(1/2)_n(1/2+2k)_n}
{(1)_n(1+k)_n}
4^{-n}
=
\dfrac{(1)_k^2}{k(1/4)_k(3/4)_k}=4^k\sqrt{\pi}
\frac{\Gamma(k)\Gamma(k+1)}
{\Gamma(2k+1/2)}.
\]
and relabeling $k$ as $x$ proves the theorem.
\end{proof}

\begin{coro}\label{coro:3log2}
The following formula is true:
\begin{align*}
  3\ln2=2+\sum_{n=1}^{\infty}
 \frac{6n+1}{2n(2n+1)}\binom{2n}{n}^{\!2}\frac1{64^n}.
\end{align*}
\end{coro}
\begin{proof}
  We first consider the right-hand side of \eqref{eq:parameter}. The term corresponding
to $n=0$ is
\[
\frac{4x+1}{x}=\frac1x+4.
\]
For $n\ge 1$, as $x\to 0^+$,
\[
\frac{6n+4x+1}{(2n+1)(n+x)}\frac{(1/2)_n(1/2+2x)_n}{(1)_n(1+x)_n}4^{-n}\longrightarrow\frac{6n+1}{n(2n+1)}\,c_n^2\,4^{-n},
\]
where $c_n=\frac{(1/2)_n}{(1)_n}=4^{-n}\binom{2n}{n}$.

Moreover, for $0\le x\le 1$, 
\[
\dfrac{6n+4x+1}{n+x}\le 7\quad \mathrm{and}\quad \dfrac{(1/2+2x)_n}{(1+x)_n}=\prod_{j=0}^{n-1}
\frac{2x+j+1/2}{x+j+1}\le 2^n
\] 
then
\[
\left|
\frac{6n+4x+1}{(2n+1)(n+x)}
\frac{(1/2)_n(1/2+2x)_n}{(1)_n(1+x)_n}
4^{-n}
\right|
\le
\frac{7c_n2^{-n}}{2n+1}.
\]
Dominated convergence gives
\[
\lim_{x\to0^+}
\sum_{n=1}^{\infty}
\frac{6n+4x+1}{(2n+1)(n+x)}
\frac{(1/2)_n(1/2+2x)_n}{(1)_n(1+x)_n}
4^{-n}
=
\sum_{n=1}^{\infty}
\frac{6n+1}{n(2n+1)}c_n^2\,4^{-n}.
\]
Therefore,
\begin{equation}\label{eq:rhs 3log2}
  \sum_{n=0}^{\infty}
\frac{6n+4x+1}{(2n+1)(n+x)}
\frac{(1/2)_n(1/2+2x)_n}{(1)_n(1+x)_n}
4^{-n}
=
\frac1x+4
+
\sum_{n=1}^{\infty}
\frac{6n+1}{n(2n+1)}c_n^2\,4^{-n}
+o(1)
\end{equation}
as $x\to0^+$.

For the left-hand side of \eqref{eq:parameter}, set
\[
A(x):=\frac{4^x\sqrt{\pi}\,\Gamma(x)\Gamma(x+1)}{\Gamma(2x+1/2)}\ \mathrm{and}\ B(x):=xA(x)=\frac{4^x\sqrt{\pi}\,\Gamma(x+1)^2}
{\Gamma(2x+1/2)}.
\]
Then $B(x)$ is regular at $x=0$, and $B(0)=1$.
Taking logarithmic derivatives gives
\[
\frac{B'(x)}{B(x)}
=
\ln 4
+2\psi(x+1)
-2\psi(2x+1/2),
\]
where $\psi(z):=(\ln\Gamma(z))'=\frac{\Gamma'(z)}{\Gamma(z)}$.
Followed by \cite[Formula 5.4.12, 5.4.13]{DLMF}, we have $\psi(1)=-\gamma$ and $\psi(1/2)=-\gamma-2\ln 2$, where $\gamma$ is the Euler constant. Then
\[
B'(0)=[\ln 4+2\psi(1)-2\psi(1/2)]B(0)=6\ln 2.
\]
Therefore,
\[
B(x)=1+6x\ln 2+O(x^2)\qquad (x\to0^+),
\]
and consequently
\begin{equation}\label{eq:lhs 3log2}
  \frac{4^x\sqrt{\pi}\,\Gamma(x)\Gamma(x+1)}
{\Gamma(2x+1/2)}
=
\frac1x+6\ln 2+O(x)
\qquad (x\to0^+).
\end{equation}

Comparing \eqref{eq:rhs 3log2} and \eqref{eq:lhs 3log2}, the formula
\[
3\ln 2=2+\sum_{n=1}^{\infty}
 \frac{6n+1}{2n(2n+1)}\binom{2n}{n}^{\!2}\frac1{64^n}.
\]
holds. This finishes the proof.
\end{proof}

\subsection{Another formula of $\zeta(3)$}
\
\newline
\indent In this section, we will prove the formula 
left in \cite[Theorem 3]{GR15}.
\begin{lem}\label{lem:Unk rep}
  $U(n,k)$ has an equivalent representation:
  \[
  U(n,k)=4^{-n-k}\frac{\Gamma(n+1/2)\Gamma(n+2k+1/2)}{\pi\,\Gamma(n+1)\Gamma(n+k+1)\Gamma(k+1)}.\]
\end{lem}
\begin{proof}
From the definition
\[
U(n,k)=4^{-n}\frac{(1/2)_n(1/2+2k)_n}{(1)_n(1+k)_n}
\frac{(1/4)_k(3/4)_k}{(1)_k^2},
\]
and the identity $(a)_m=\Gamma(a+m)/\Gamma(a)$, we obtain
\[
U(n,k)=
4^{-n}
\frac{\Gamma(n+1/2)\Gamma(n+2k+1/2)}{\Gamma(n+1)\Gamma(n+k+1)\Gamma(k+1)}
R_k,
\]
where
\[
R_k=
\frac{\Gamma(k+1/4)\Gamma(k+3/4)}{\Gamma(1/2)\Gamma(2k+1/2)
 \Gamma(1/4)\Gamma(3/4)}.
\]

By Legendre's duplication formula (see \cite[Formula 5.5.5]{DLMF}), $\Gamma(z)\Gamma\left(z+\frac12\right)=2^{1-2z}\sqrt{\pi}\,\Gamma(2z)$. Let $z=k+\frac14$, we have
\[
\Gamma\left(k+\frac14\right)
\Gamma\left(k+\frac34\right)=
2^{1/2-2k}\sqrt{\pi}
\Gamma\left(2k+\frac12\right).
\]
Moreover, since $\Gamma\left(\frac12\right)=\sqrt{\pi}$ and $\Gamma\left(\frac14\right)
\Gamma\left(\frac34\right)=
\frac{\pi}{\sin(\pi/4)}=
\pi\sqrt{2}$,
\[
R_k=\frac{4^{-k}}{\pi}.
\]
Substituting this into the preceding expression yields
\[
U(n,k)
=
4^{-n-k}
\frac{\Gamma(n+1/2)\Gamma(n+2k+1/2)}
{\pi\,\Gamma(n+1)\Gamma(n+k+1)\Gamma(k+1)},
\]
as required.
\end{proof}
\begin{coro}\label{coro:U(n+1/2,k)}
For every integer $n\ge 0$ and $k\in\mathbb{R}_{\geq0}$,
\[
U\left(n+\frac12,k\right)=
\frac{2}{\pi^2(2k+1)}\,4^{-n}
\frac{(1)_n(1+2k)_n}
     {(3/2)_n(k+3/2)_n}.
\]
\end{coro}
\begin{proof}
By the preceding lemma,
\[
U\left(n+\frac12,k\right)=
4^{-n-k-1/2}
\frac{\Gamma(n+1)\Gamma(n+2k+1)}{\pi\,\Gamma(n+3/2)\Gamma(n+k+3/2)\Gamma(k+1)}.
\]
Using $(a)_n=\Gamma(a+n)/\Gamma(a)$, we obtain
\[
U\left(n+\frac12,k\right)=
4^{-n-k-1/2}
\frac{\Gamma(2k+1)}{\pi\,\Gamma(3/2)\Gamma(k+3/2)\Gamma(k+1)}
\frac{(1)_n(1+2k)_n}{(3/2)_n(k+3/2)_n}.
\]
Since $\Gamma\left(\frac32\right)=\frac{\sqrt{\pi}}{2}$ and Legendre's duplication formula (see \cite[Formula 5.5.5]{DLMF}), applied with $z=k+1$, we have:
\[
\Gamma(k+1)\Gamma\left(k+\frac32\right)=
2^{-2k-1}\sqrt{\pi}\,\Gamma(2k+2)=2^{-2k-1}(2k+1)\sqrt{\pi}\,\Gamma(2k+1).
\]
Then
\[
U\left(n+\frac12,k\right)=
\frac{2}{\pi^2(2k+1)}\,4^{-n}
\frac{(1)_n(1+2k)_n}
     {(3/2)_n(k+3/2)_n},
\]
as claimed.
\end{proof}

\begin{thm}\label{thm:zeta3}
  The following formula is true:
  \begin{align*}
 \zeta(3)=\frac47\sum_{n=0}^{\infty}
 \frac{(3n+2)4^n}{(2n+1)^3(n+1)\binom{2n}{n}^{2}}.
\end{align*}
\end{thm}
\begin{proof}
We use the WZ pair in \eqref{eq:kernel}, \eqref{eq:pair} and shift the WZ identity \eqref{eq:wz} by $n\mapsto n+\frac12$ and $k\mapsto k+y, y\ge 0$. Summing the equation over $n\ge 0$ and using the fact $\lim_{n\to\infty}
F\left(n+\frac12,k+y\right)=0$, we have
\[
F\left(\frac12,k+y\right)
=
\sum_{n=0}^{\infty}
G\left(n+\frac12,k+y\right)-\sum_{n=0}^{\infty}
G\left(n+\frac12,k+y+1\right).
\]
Sum both sides of the equation from $k=0$ to $k=\infty$:
\begin{equation}\label{eq:k sum}
  \sum_{k=0}^{\infty}
F\left(\frac12,k+y\right)
=
\sum_{n=0}^{\infty}
G\left(n+\frac12,y\right)-\lim_{k\to\infty}\sum_{n=0}^{\infty}
G\left(n+\frac12,k+y\right).
\end{equation}

We first verify the limit of the summation of $G(n+\frac12,k+y)$ as $k\to\infty$. By Corollary \ref{coro:U(n+1/2,k)},
\[
U\left(n+\frac12,k\right)
=
\frac{2}{\pi^2(2k+1)}\,4^{-n}
\frac{(1)_n(1+2k)_n}
     {(3/2)_n(k+3/2)_n}.
\]
Since
\[
\frac{(1+2k)_n}{(k+3/2)_n}=\prod_{j=0}^{n-1}\frac{2k+1+j}{k+3/2+j} \le 2^n\quad \mathrm{and}\quad \frac{(1)_n}{(3/2)_n}\le 1,
\]
we have
\[
0<
U\left(n+\frac12,k\right)
\le
\frac{2}{\pi^2(2k+1)}2^{-n}.
\]
Moreover,
\[
G\left(n+\frac12,k\right)
=
U\left(n+\frac12,k\right)
\frac{k(3n+2k+2)}
     {2(n+k+1/2)(n+1)}.
\]
Since
\[
\frac{3n+2k+2}{n+k+1/2}\le 4,
\]
it follows that
\[
0\le
G\left(n+\frac12,k\right)
\le
\frac{2}{\pi^2}
\frac{2^{-n}}{n+1}.
\]
Thus $\sum_{n=0}^{\infty}
G\left(n+\frac12,k\right)$ converges uniformly for $k\ge 0$. For each fixed $n$, the limit of
$G(n+\frac12,k)$ as $k\to\infty$ exists. Hence, by the above
summable majorant and dominated convergence, denote
\[
\sum_{n=0}^{\infty}
G\left(n+\frac12,k\right)\longrightarrow L
\qquad (k\to\infty)
\]
for some finite constant $L$.

Now followed by the equation \eqref{eq:k sum} and $L$ is independent of $y$, we obtain
\begin{equation}\label{eq: sum k+y}
  \sum_{k=0}^{\infty}
F\left(\frac12,k+y\right)
=
\sum_{n=0}^{\infty}
G\left(n+\frac12,y\right)-L.
\end{equation}
At $y=0$, the explicit factor $k$ in $G(n+1/2,k)$ gives $G\left(n+\frac12,0\right)=0$ and hence the summation is zero. Therefore, subtracting the identity
\eqref{eq: sum k+y} at $y=0$ from \eqref{eq: sum k+y}, we have
\[
\sum_{k=0}^{\infty}
\left[
F\left(\frac12,k+y\right)
-
F\left(\frac12,k\right)
\right]
=
\sum_{n=0}^{\infty}
G\left(n+\frac12,y\right).
\]
Dividing by $y>0$ gives
\begin{equation}\label{eq: differential y}
  \sum_{k=0}^{\infty}
\frac{
F(1/2,k+y)-F(1/2,k)
}{y}
=
\sum_{n=0}^{\infty}
\frac{G(n+1/2,y)}{y}.
\end{equation}

We next justify taking the limit $y\to0^+$ in \eqref{eq: differential y}, by Corollary \ref{coro:U(n+1/2,k)},
\begin{equation}\label{eq:derivative F(1/2,k)}
 U\left(\frac12,k\right)
=
\frac{2}{\pi^2(2k+1)}\qquad F\left(\frac12,k\right)
=
-\frac{1}{\pi^2(2k+1)^2}\quad\mathrm{and}\quad \partial_k F(\frac{1}{2},k)=\frac{4}{\pi^2(2k+1)^3}. 
\end{equation}
Then the mean value theorem implies, for $0<y\le1$,
\[
\left|
\frac{F(1/2,k+y)-F(1/2,k)}{y}
\right|
\le
\frac{4}{\pi^2(2k+1)^3}.
\]
The majorant on the right is summable.

On the other hand,
\[
\frac{G(n+1/2,y)}{y}
=
U\left(n+\frac12,y\right)
\frac{3n+2y+2}
{2(n+y+1/2)(n+1)}.
\]
Using the estimates obtained above,
\[
U\left(n+\frac12,y\right)
\le
\frac{2}{\pi^2(2y+1)}2^{-n}
\le
\frac{2}{\pi^2}2^{-n},\qquad\frac{3n+2y+2}{n+y+1/2}\le4,
\]
we obtain
\[
0\le
\frac{G(n+1/2,y)}{y}
\le
\frac{4}{\pi^2}
\frac{2^{-n}}{n+1}.
\]
Again, the majorant is summable. Therefore, by dominated convergence,
letting $y\to0^+$ in \eqref{eq: differential y} yields
\[
\sum_{k=0}^{\infty}
\partial_kF\left(\frac12,k\right)
=
\sum_{n=0}^{\infty}
\partial_kG\left(n+\frac12,0\right).
\]

For the left-hand side, we have:
\[
\sum_{k=0}^{\infty}
\partial_kF\left(\frac12,k\right)=
\frac{4}{\pi^2}
\sum_{k=0}^{\infty}\frac{1}{(2k+1)^3} =
\frac{4}{\pi^2}
\left(1-\frac{1}{2^3}\right)\zeta(3) =
\frac{7\zeta(3)}{2\pi^2}.
\]
For the right-hand side, the explicit factor $k$ gives:
\begin{equation}\label{eq:derivative G(n+1/2,0)}
  \partial_kG\left(n+\frac12,0\right)=
U\left(n+\frac12,0\right)
\frac{3n+2}{(2n+1)(n+1)}
\end{equation}
and Corollary \ref{coro:U(n+1/2,k)} gives
\[
U\left(n+\frac12,0\right)=
\frac{2}{\pi^2}4^{-n}
\frac{(1)_n^2}{(3/2)_n^2}=
\frac{2\cdot4^n}{\pi^2(2n+1)^2\binom{2n}{n}^2}
\]

Therefore
\[
\frac{7\zeta(3)}{2\pi^2}=
\frac{2}{\pi^2}
\sum_{n=0}^{\infty}
\frac{(3n+2)4^n}{(2n+1)^3(n+1)\binom{2n}{n}^2}
\]
i.e.,
\[
\zeta(3)=
\frac{4}{7}
\sum_{n=0}^{\infty}
\frac{(3n+2)4^n}{(2n+1)^3(n+1)\binom{2n}{n}^2},
\]
which proves the theorem.
\end{proof}

%% file: sections/section2.tex
\section{The $2\ln 2$ model and the congruence problem}

We now turn to the arithmetic properties of finite truncations of the logarithmic series. Recall that Guillera and Rogers \cite[Theorem 1]{GR15} established the
identity
\[
2\ln 2=1+\sum_{n=1}^{\infty}\frac{4n+1}{2n(2n+1)}
\binom{2n}{n}^{2}\frac{1}{2^{4n}}.
\]
Letting $p$ be an odd prime, we write
\[
a_n=\frac{4n+1}{2n(2n+1)}
\binom{2n}{n}^{2}\frac{1}{2^{4n}}.
\]
Since the denominator of $a_n$ is divisible by $p$ precisely
when $n=\dfrac{p-1}{2}$, among $1\leq n\leq p-1$, it is natural to separate the truncated sum into three parts:
\[
\sum_{n=1}^{p-1}a_n
=
\sum_{n=1}^{\frac{p-3}{2}}a_n
+a_{\frac{p-1}{2}}
+\sum_{n=\frac{p+1}{2}}^{p-1}a_n.
\]

\begin{thm}\label{thm:2log2 congruences}
Let $p$ be an odd prime. Then
\begin{equation}\label{eq:2log2-lower}
\sum_{n=1}^{\frac{p-3}{2}}
\frac{4n+1}{2n(2n+1)}
\binom{2n}{n}^{2}\frac{1}{2^{4n}}
\equiv 0\pmod p,
\end{equation}
and
\begin{equation}\label{eq:2log2-upper}
\sum_{n=\frac{p+1}{2}}^{p-1}
\frac{4n+1}{2n(2n+1)}
\binom{2n}{n}^{2}\frac{1}{2^{4n}}
\equiv 0\pmod{p^3}.
\end{equation}
\end{thm}
\begin{proof}
  We will use the following congruences of
\cite[Theorems 1.1 and 1.2]{Sun14}. For every prime \(p>5\),
\begin{align}
\sum_{n=1}^{(p-1)/2}
\frac{\binom{2n}{n}^{2}}{n16^n}
&\equiv
-2H_{(p-1)/2}-\frac{7}{2}p^2B_{p-3}
\pmod{p^3},
\label{eq:sun1}\\
\sum_{n=(p+1)/2}^{p-1}
\frac{\binom{2n}{n}^{2}}{n16^n}
&\equiv
\frac{7}{2}p^2B_{p-3}
\pmod{p^3},
\label{eq:sun2}\\
\sum_{n=0}^{(p-3)/2}
\frac{\binom{2n}{n}^{2}}{(2n+1)16^n}
&\equiv
-2q_p(2)-pq_p(2)^2
+\frac{5}{12}p^2B_{p-3}
\pmod{p^3},
\label{eq:sun3}\\
\sum_{n=(p+1)/2}^{p-1}
\frac{\binom{2n}{n}^{2}}{(2n+1)16^n}
&\equiv
-\frac{7}{4}p^2B_{p-3}
\pmod{p^3}.
\label{eq:sun4}
\end{align}
Here
\[
H_m=\sum_{k=1}^{m}\frac1k,
\qquad
q_p(2)=\frac{2^{p-1}-1}{p}.
\]

For the lower truncation, using
\[
a_n=\frac{4n+1}{2n(2n+1)}
\binom{2n}{n}^{2}\frac{1}{2^{4n}}=
\frac{1}{2}
\frac{\binom{2n}{n}^{2}}{n16^n}
+
\frac{\binom{2n}{n}^{2}}{(2n+1)16^n},
\]
we obtain
\[
\sum_{n=1}^{(p-3)/2}a_n
=
\frac12
\sum_{n=1}^{(p-3)/2}
\frac{\binom{2n}{n}^{2}}{n16^n}
+
\sum_{n=1}^{(p-3)/2}
\frac{\binom{2n}{n}^{2}}{(2n+1)16^n}.
\]
By \eqref{eq:sun1},
\[
\sum_{n=1}^{(p-3)/2}
\frac{\binom{2n}{n}^{2}}{n16^n}
\equiv
-2H_{(p-1)/2}
-
\frac{\binom{p-1}{(p-1)/2}^{2}}{4^{p-1}(p-1)/2}
\pmod p.
\]
Since $\binom{p-1}{(p-1)/2}
\equiv(-1)^{(p-1)/2}\pmod p,
4^{(p-1)}\equiv1\pmod p$ and $\frac{1}{(p-1)/2}\equiv-2\pmod p$,
\begin{equation}\label{eq:2log2-lower-left}
\frac12
\sum_{n=1}^{(p-3)/2}
\frac{\binom{2n}{n}^{2}}{n16^n}
\equiv
-H_{(p-1)/2}+1
\pmod p.  
\end{equation}
On the other hand, \eqref{eq:sun3} gives
\begin{equation}\label{eq:2log2-lower-right}
 \sum_{n=1}^{(p-3)/2}
\frac{\binom{2n}{n}^{2}}{(2n+1)16^n}
\equiv
-2q_p(2)-1
\pmod p, 
\end{equation}
where the term \(n=0\), which is equal to \(1\), has been removed. \cite[Formula (45)]{Leh38} implies,
\[
H_{(p-1)/2}=\sum_{k=1}^{(p-1)/2}\frac1k\equiv-2q_p(2)\pmod p,
\]
Combining \eqref{eq:2log2-lower-left} and \eqref{eq:2log2-lower-right}, we conclude that
\[
\sum_{n=1}^{(p-3)/2}\frac{4n+1}{2n(2n+1)}
\binom{2n}{n}^{2}\frac{1}{2^{4n}}
\equiv
-H_{(p-1)/2}-2q_p(2)
\equiv0\pmod p.
\]
This proves \eqref{eq:2log2-lower}.

For the upper truncation, \eqref{eq:sun2} and
\eqref{eq:sun4} give
\begin{align*}
\sum_{n=(p+1)/2}^{p-1}\frac{4n+1}{2n(2n+1)}
\binom{2n}{n}^{2}\frac{1}{2^{4n}}
&=
\frac12
\sum_{n=(p+1)/2}^{p-1}
\frac{\binom{2n}{n}^{2}}{n16^n}
+
\sum_{n=(p+1)/2}^{p-1}
\frac{\binom{2n}{n}^{2}}{(2n+1)16^n}\\
&\equiv
\frac12\left(\frac72p^2B_{p-3}\right)
-\frac74p^2B_{p-3}\\
&=0
\pmod{p^3}.
\end{align*}
This proves \eqref{eq:2log2-upper}.

The cases $p=3$ and $p=5$ follow by direct computation.
\end{proof}

\begin{rem}
It is natural to ask whether the upper-truncation congruence in \eqref{eq:2log2-upper} can be strengthened to a congruence modulo \(p^4\). 
A numerical computation suggests that this phenomenon is highly
exceptional. Indeed, among primes \(p<300\,000\,000\), the congruence
\[
\sum_{n=(p+1)/2}^{p-1}
\frac{4n+1}{2n(2n+1)}
\binom{2n}{n}^{2}\frac{1}{2^{4n}}
\equiv0\pmod{p^4}
\]
holds only for
\[
p=3,\qquad p=1237.
\]
This suggests that the congruence modulo \(p^4\) does not hold
uniformly and that the primes satisfying it may be related to a
deeper arithmetic phenomenon. Moreover, for both of these primes, the corresponding upper truncation
is nonzero modulo \(p^5\). 

\noindent\textbf{Conjecture:} For every odd prime \(p\),
\begin{equation}
\sum_{n=(p+1)/2}^{p-1}
\frac{4n+1}{2n(2n+1)}
\binom{2n}{n}^{2}\frac{1}{2^{4n}}
\not\equiv 0 \pmod{p^5}.\label{conj}
\end{equation}
\end{rem}

\begin{rem}
The upper-truncation congruence in \eqref{eq:2log2-upper} is stronger than what is implied by the divisibility of the individual summands. It is also worth noting that the full truncation behaves quite
differently. If denoting
\[
S:=
\sum_{n=1}^{p-1}
\frac{4n+1}{2n(2n+1)}
\binom{2n}{n}^{2}\frac{1}{2^{4n}},
\]
then
\[
v_p\bigl(S\bigr)=-1.
\]
Indeed, if
\[
a_n=
\frac{4n+1}{2n(2n+1)}
\binom{2n}{n}^{2}\frac{1}{2^{4n}},
\]
then \(a_n\in\mathbb Z_p\) for
\(n\neq (p-1)/2\), whereas
\[
v_p\left(a_{(p-1)/2}\right)=-1.
\]
Consequently, the middle term is the unique term of negative
\(p\)-adic valuation in the full truncation, which further explains
the natural separation into the lower, middle, and upper ranges.
\end{rem}

%% file: sections/section3.tex
\section{The $3\ln 2$ model and the congruence problem}

The congruences for the truncated series associated with $2\ln 2$, recalled in the preceding section, naturally
motivate the investigation of the corresponding series for
$3\ln 2$. We now turn to the identity
\[
3\ln 2=2+\sum_{n=1}^{\infty}\frac{6n+1}{2n(2n+1)}
\binom{2n}{n}^{2}\frac{1}{2^{6n}}.
\]

For an odd prime $p$, define
\[
b_n=\frac{6n+1}{2n(2n+1)}
\binom{2n}{n}^{2}\frac{1}{2^{6n}}.
\]
As in the $2\ln 2$ case, we consider the two truncations
\[
\sum_{n=1}^{\frac{p-3}{2}}b_n
\qquad\text{and}\qquad
\sum_{n=\frac{p+1}{2}}^{p-1}b_n.
\]

We write
\[
q_p(2)=\frac{2^{p-1}-1}{p}
\]
for the Fermat quotient of $2$ with respect to $p$.
The Bernoulli numbers $B_n$ are defined by the generating
function
\[
\frac{t}{e^t-1}
=\sum_{n=0}^{\infty}B_n\frac{t^n}{n!}.
\]
In particular, $B_0=1$ and $B_1=-1/2$.

\begin{thm}\label{thm:3log2 congruence}
For every prime $p\geq5$,
\begin{align}
 \sum_{n=1}^{\frac{p-3}{2}}
 \frac{6n+1}{2n(2n+1)}\binom{2n}{n}^{\!2}\frac1{2^{6n}}
 &\equiv q_p(2)-1\pmod p,
 \label{eq:3log2 lower}\\
 \sum_{n=\frac{p+1}{2}}^{p-1}
 \frac{6n+1}{2n(2n+1)}\binom{2n}{n}^{\!2}\frac1{2^{6n}}
 &\equiv-\frac7{16}p^2B_{p-3}\pmod{p^3}.
 \label{eq:3log2 upper}
\end{align}
\end{thm}
The rest of this section is devoted to the proof of the congruent identitites \eqref{eq:3log2 lower} and \eqref{eq:3log2 upper}.

\subsection{The lower truncation of $3\ln 2$.}
\
\newline
\indent We first establish the congruence for the lower truncation
associated with $3\ln 2$. The main ingredient is a terminating
identity obtained from the WZ pair constructed in Section 2:
\begin{align*}
 &U(n,k)=4^{-n}\frac{(1/2)_n(1/2+2k)_n}{(1)_n(1+k)_n}
                  \frac{(1/4)_k(3/4)_k}{(1)_k^2}\\
 F(n,k)&=-\frac{n}{2(n+k)}U(n,k),\qquad
 G(n,k)=\frac{k(6n+4k+1)}{2(n+k)(2n+1)}U(n,k).
\end{align*}

\begin{lem}\label{lem:terminating-identity}
Let $H(n,x)$ be defined as:
\begin{align*}
    H(n,x)=4^{-n}\frac{(1/2)_n(1/2+2x)_n}{(1)_n(1+x)_n}\frac{(6n+4x+1)}{(n+x)(2n+1)}.
\end{align*}
Then for every integer $N\geq0$, we have
\begin{equation}\label{eq:terminating-identity}
\sum_{n=0}^{N}H(n,-\frac{2N+1}{4})=\sum_{n=0}^{N}4^{-n}\frac{(1/2)_n(-N)_n}{(1)_n((3-2N)/4)_n}
\left(
\frac{1}{n-(2N+1)/4}
+\frac{4}{2n+1}
\right)=0.
\end{equation}
\end{lem}

\begin{proof}
Summing the WZ relation $F(n+1,x)-F(n,x)=G(n,x+1)-G(n,x)$ over $0\leq n\leq N$, we obtain
\begin{equation}\label{eq:N-wz-sum}
F(N+1,x)=
\sum_{n=0}^{N}G(n,x+1)
-\sum_{n=0}^{N}G(n,x).
\end{equation}
By the definition of $H(n,x)$, we have
\[
G(n,x)=\frac{x}{2}
\frac{(1/4)_x(3/4)_x}{(1)_x^2}H(n,x).
\]
Moreover,
\[
\frac{
(1/4)_{x+1}(3/4)_{x+1}/(1)_{x+1}^2
}{(1/4)_x(3/4)_x/(1)_x^2}=
\frac{(4x+1)(4x+3)}{16(x+1)^2}.
\]
Substituting these expressions into
\eqref{eq:N-wz-sum} and cancelling the common
factor yields
\begin{equation}
\frac{(4x+1)(4x+3)}{16(x+1)}
\sum_{n=0}^{N}H(n,x+1)
-x\sum_{n=0}^{N}H(n,x)=-\frac{N+1}{N+1+x}4^{-(N+1)}
\frac{(1/2)_{N+1}(1/2+2x)_{N+1}}
{(1)_{N+1}(1+x)_{N+1}}.
\label{eq:N-wz-identity}
\end{equation}
This identity is initially obtained for generic $x$ and
extends as a rational identity wherever the displayed
expressions are defined.

Now set
\[
x_N=-\frac{2N+1}{4}.
\]
Since$\left(\frac12+2x_N\right)_{N+1}=(-N)_{N+1}=0$,
the right-hand side of \eqref{eq:N-wz-identity}
vanishes.

Now we prove $\displaystyle\sum_{n=0}^N H(n,x_N)=0$ by induction. For $N=0,1$, the coefficient
$(4x_N+1)(4x_N+3)$ also vanishes. Since $x_N\neq0$,
we immediately obtain
\[
\sum_{n=0}^{N}H(n,x_N)=0\quad(N=0,1).
\]
For $N\geq2$, assuming that the identity holds for $N-2$, observe that
\[
x_N+1=x_{N-2},
\qquad
\frac12+2(x_N+1)=-(N-2).
\]
Because the term $(1/2+2(x_N+1))_n$ appears in $H(n,x_N+1)$, $H(n,x_N+1)=0$ for $n\geq N-1$. Then
\[
\sum_{n=0}^{N}H(n,x_N+1)
=
\sum_{n=0}^{N-2}H(n,x_{N-2}).
\]
By the induction hypothesis, the latter sum vanishes. Substituting into
\eqref{eq:N-wz-identity} gives
\[
x_N\sum_{n=0}^{N}H(n,x_N)=0.
\]
Since $x_N\neq0$, the desired identity follows.
\end{proof}

We now apply the preceding Lemma to obtain the desired
congruence for the lower truncation.

\textbf{Proof of Formula \eqref{eq:3log2 lower} in Theorem \ref{thm:3log2 congruence}:}
Applying Lemma~\ref{lem:terminating-identity} with $N=\frac{p-1}{2}$, we obtain
\begin{equation}\label{eq:lower-specialization}
\sum_{n=0}^{\frac{p-1}{2}}4^{-n}\frac{(1/2)_n(-(p-1)/{2})_n}{(1)_n((4-p)/4)_n}
\left(
\frac{1}{n-(p/4)}
+\frac{4}{2n+1}
\right)=0.
\end{equation}
For convenience, denote the $n$th summand by
\[
T_n=
4^{-n}\frac{(1/2)_n(-(p-1)/{2})_n}{(1)_n((4-p)/4)_n}
\left(
\frac{1}{n-(p/4)}
+\frac{4}{2n+1}
\right).
\]

First, when $n=0$,
\begin{equation}\label{eq:lower-first-term}
T_0=-\frac4p+4.
\end{equation}

Next, for $1\leq n\leq \frac{p-3}{2}$, since$-(p-1)/2\equiv1/2\pmod p$ and $(4-p)/4\equiv1\pmod p$, we have
\[
4^{-n}\frac{(1/2)_n(-(p-1)/{2})_n}{(1)_n((4-p)/4)_n}
\equiv
4^{-n}\frac{(1/2)_n^2}{((1)_n)^2}
=
\binom{2n}{n}^{2}\frac1{2^{6n}}
\pmod p.
\]
Furthermore,
\[
\frac{1}{n-\frac p4}+\frac4{2n+1}
\equiv
\frac{6n+1}{n(2n+1)}
\pmod p.
\]
It follows that
\begin{equation}\label{eq:lower-middle-term}
\sum_{n=1}^{\frac{p-3}{2}}T_n
\equiv
2\sum_{n=1}^{\frac{p-3}{2}}
\frac{6n+1}{2n(2n+1)}
\binom{2n}{n}^{2}\frac1{2^{6n}}=2\sum_{n=1}^{\frac{p-3}{2}}b_n
\pmod p.
\end{equation}

It remains to evaluate the last term. By definition,
\[
T_{\frac{p-1}{2}}
=4^{-\frac{p-1}{2}}
\frac{8(p-1)}{p(p-2)}
\frac{(-1)^{\frac{p-1}{2}}
(1/2)_{\frac{p-1}{2}}}
{((4-p)/4)_{\frac{p-1}{2}}}.
\]
Observe that
\[
4^{-\frac{p-1}{2}}
\frac{(-1)^{\frac{p-1}{2}}
(1/2)_{\frac{p-1}{2}}}
{((4-p)/4)_{\frac{p-1}{2}}}
=(-1)^{\frac{p-1}{2}}\binom{p-1}{(p-1)/2}
\frac{\prod_{j=1}^{(p-1)/2}
\left(1-\frac{p}{4j}\right)^{-1}}{4^{p-1}}.
\]
By Morley's congruence \cite{Mor94},
\[
(-1)^{\frac{p-1}{2}}
\binom{p-1}{(p-1)/2}
\equiv 4^{p-1}\pmod{p^3}.
\]
Moreover, by the elementary expansion
\[
\prod_{j=1}^{(p-1)/2}
\left(1-\frac{p}{4j}\right)^{-1}
\equiv
1+\frac{p}{4}
\sum_{j=1}^{(p-1)/2}\frac1j
\pmod{p^2}.
\]
\cite[Formula (45)]{Leh38} implies
\[
\sum_{j=1}^{(p-1)/2}\frac1j
\equiv -2q_p(2)\pmod p.
\]
Then
\begin{align*}
\prod_{j=1}^{(p-1)/2}
\left(1-\frac{p}{4j}\right)^{-1}
\equiv
1-\frac p2q_p(2)
\pmod{p^2}.
\end{align*}
Consequently,
\[
T_{\frac{p-1}{2}}=
\frac{8(p-1)}{p(p-2)}
\left(1-\frac p2q_p(2)+O(p^2)\right)=\frac4p-2-2q_p(2)+O(p).
\]
i.e.,
\begin{equation}\label{eq:lower-last-term}
    T_{\frac{p-1}{2}}\equiv\frac4p-2-2q_p(2) \pmod p.
\end{equation}

Finally, followed by \eqref{eq:lower-specialization}, \eqref{eq:lower-first-term}, \eqref{eq:lower-middle-term} and \eqref{eq:lower-last-term}, we have
\begin{align*}
\sum_{n=1}^{\frac{p-3}{2}}
\frac{6n+1}{2n(2n+1)}
\binom{2n}{n}^{2}\frac{1}{2^{6n}}
\equiv q_p(2)-1\pmod p.
\end{align*}
This finishes the proof. \hfill$\square$.

\subsection{The upper truncation of $3\ln 2$}
\
\newline
\indent For $\frac{p+1}{2}\leq n\leq p-1$, $p|\binom{2n}{n}$, then
\[
\sum_{n=\frac{p+1}{2}}^{p-1}
 \frac{6n+1}{2n(2n+1)}\binom{2n}{n}^{\!2}\frac1{2^{6n}}
\equiv 0\pmod{p^2}.
\]
Therefore we only need to determine
\[
\frac{1}{p^2}
\sum_{n=\frac{p+1}{2}}^{p-1}
 \frac{6n+1}{2n(2n+1)}\binom{2n}{n}^{\!2}\frac1{2^{6n}}
\pmod p.
\]
In contrast to the corresponding congruence for
$2\ln 2$, this normalized sum does not vanish in general.
We shall show that its residue can be expressed in terms
of the Bernoulli number $B_{p-3}$.

The proof relies on a finite inverse-binomial-square
identity obtained from the half-integer shift method of the
WZ pair used in the proof of the formula of $\zeta(3)$  in Section 2.
\subsubsection{A finite identity derived from the WZ Pair}
\begin{lem}\label{lem:finite-derived-wz}
For every integer $N\geq1$, we have
\begin{align}
\sum_{n=0}^{N-1}
\frac{(3n+2)4^n}
{(2n+1)^3(n+1)\binom{2n}{n}^2}=
2\sum_{j=0}^{N-1}\frac1{(2j+1)^3}
+\sum_{\ell=1}^{\lfloor N/2\rfloor}
\frac{(-1)^{\ell+1}
\binom N\ell\binom{2N-2\ell}{N}}
{2\ell\binom{2N}{N}}
\sum_{j=N-\ell}^{N-1}\frac1{(2j+1)^2}.
\label{eq:lem:finite-derived-wz}
\end{align}
\end{lem}

\begin{proof}
Recall the WZ pair $(F,G)$ constructed in Section 2:
\begin{align*}
 &U(n,k)=4^{-n}\frac{(1/2)_n(1/2+2k)_n}{(1)_n(1+k)_n}
                  \frac{(1/4)_k(3/4)_k}{(1)_k^2}\\
 F(n,k)&=-\frac{n}{2(n+k)}U(n,k),\qquad
 G(n,k)=\frac{k(6n+4k+1)}{2(n+k)(2n+1)}U(n,k).
\end{align*}
For every integer $N\geq0$, followed by Corollary \ref{coro:U(n+1/2,k)}, its half-integer shift gives
\begin{equation}\label{eq:F(N+1/2,k)}
   \begin{aligned}
\pi^2F\left(N+\frac12,k\right)&=-\frac{2N+1}{(2N+1+2k)(2k+1)}4^{-N}\frac{(1)_N(1+2k)_N}{(\frac{3}{2})_N(k+\frac{3}{2})_N}\\
&=-\frac1{\binom{2N}{N}}
\frac{(1+2k)_N}
{(2k+1)(2k+2N+1)(k+\frac32)_N}.
\end{aligned} 
\end{equation}
The second equality holds since 
\[
(2N+1)4^{-N}\frac{(1)_N}{(\frac32)_N}
=\frac1{\binom{2N}{N}}.
\]

Applying the WZ relation with $n$ replaced by $n+\frac12$
and summing over $0\leq n\leq N-1$, we obtain
\begin{align*}
F\left(N+\frac12,k\right)-F\left(\frac12,k\right)
=
\sum_{n=0}^{N-1}
\left[
G\left(n+\frac12,k+1\right)
-G\left(n+\frac12,k\right)
\right].
\end{align*}
Since the summation is finite, we may differentiate it with respect to the second variable. Thus
\[
\frac{\partial}{\partial 2}
F\left(N+\frac12,k\right)
-
\frac{\partial}{\partial 2}
F\left(\frac12,k\right)
=
\sum_{n=0}^{N-1}
\left[
\frac{\partial}{\partial 2}
G\left(n+\frac12,k+1\right)
-
\frac{\partial}{\partial 2}
G\left(n+\frac12,k\right)
\right].
\]
Here $\frac{\partial}{\partial 2}
F\left(- ,k\right)$ denotes $\frac{\partial}{\partial k}
F\left(- ,k\right)$.
Summing this identity over $k=0,\ldots,M$, we have
\begin{align*}
\sum_{n=0}^{N-1}
\frac{\partial}{\partial 2}
G\left(n+\frac12,0\right)=
\sum_{k=0}^{M}
\left[
\frac{\partial}{\partial 2}
F\left(\frac12,k\right)-
\frac{\partial}{\partial 2}
F\left(N+\frac12,k\right)
\right]+
\sum_{n=0}^{N-1}
\frac{\partial}{\partial 2}
G\left(n+\frac12,M+1\right).
\end{align*}
For each fixed $n$, the last term tends to zero as
$M\to\infty$. Therefore
\begin{equation}\label{eq:finite-differentiated-WZ-limit}
\sum_{n=0}^{N-1}
\frac{\partial}{\partial 2}
G\left(n+\frac12,0\right)
=
\sum_{k=0}^{\infty}
\left[
\frac{\partial}{\partial 2}
F\left(\frac12,k\right)
-
\frac{\partial}{\partial 2}
F\left(N+\frac12,k\right)
\right].
\end{equation}
Formulae \eqref{eq:derivative G(n+1/2,0)} and \eqref{eq:derivative F(1/2,k)} in Section 2 give
\[
\frac{\partial}{\partial 2}G\left(n+\frac12,0\right)=
\frac{2}{\pi^2}
\frac{(3n+2)4^n}
{(2n+1)^3(n+1)\binom{2n}{n}^2}\quad\mathrm{and}\quad\
\frac{\partial}{\partial 2}F\left(\frac12,k\right)=
\frac{4}{\pi^2(2k+1)^3}.
\]
On the other hand, by \eqref{eq:F(N+1/2,k)},
\begin{equation}
\frac{\partial}{\partial 2}
F\left(N+\frac12,k\right)
=
-\frac1{\pi^2\binom{2N}{N}}
\frac{d}{dk}
\left(
\frac{(1+2k)_N}
{(2k+1)(2k+2N+1)(k+\frac32)_N}
\right).
\nonumber
\end{equation}
Substituting these three evaluations into \eqref{eq:finite-differentiated-WZ-limit} and
multiplying both sides by $\pi^2$, we obtain
\begin{align}
2\sum_{n=0}^{N-1}
\frac{(3n+2)4^n}
{(2n+1)^3(n+1)\binom{2n}{n}^2}=
\sum_{k=0}^{\infty}
\Bigg[
\frac4{(2k+1)^3}+
\frac1{\binom{2N}{N}}
\frac{d}{dk}
\left(
\frac{(1+2k)_N}
{(2k+1)(2k+2N+1)(k+\frac32)_N}
\right)
\Bigg].
\label{eq:finite-WZ-derivative}
\end{align}

To evaluate the remaining derivative term, put
\[
R_N(k)=
\frac1{\binom{2N}{N}}
\frac{(1+2k)_N}
{(2k+1)(2k+2N+1)(k+\frac32)_N}.
\]
We first simplify the rational function $R_N(k)$. Since
\[
(1+2k)_N=\prod_{r=0}^{N-1}(2k+r+1)\quad\mathrm{and}\quad\left(k+\frac32\right)_N
=
2^{-N}\prod_{r=0}^{N-1}(2k+2r+3),
\]
cancelling the common factors gives
\begin{equation}
R_N(k)
=
\frac{2^N}{\binom{2N}{N}}
\frac{
\displaystyle\prod_{j=1}^{\lfloor N/2\rfloor}(2k+2j)
}{
(2k+2N+1)^2
\displaystyle\prod_{\ell=1}^{\lfloor N/2\rfloor}
(2k+2N-2\ell+1)
}.\nonumber
\end{equation}
Thus $R_N(k)$ has a double pole at
\[
2k+2N+1=0
\]
and simple poles at
\[
2k+2N-2\ell+1=0,
\qquad
1\leq\ell\leq\lfloor N/2\rfloor.
\]
Moreover, $R_N(k)=O(k^{-2})$ as $k\to\infty$.
Hence its partial fraction decomposition is of the form:
\[
R_N(k)
=
\frac{1}{(2k+2N+1)^2}-
\sum_{\ell=1}^{\lfloor N/2\rfloor}
A_\ell
\left(
\frac{1}{2k+2N-2\ell+1}
-\frac{1}{2k+2N+1}
\right).
\]
By taking residues at the simple poles
$2k+2N-2\ell+1=0$, we obtain
\[
A_\ell
=
\frac{(-1)^{\ell+1}
\binom N\ell\binom{2N-2\ell}{N}}
{2\ell\binom{2N}{N}}.
\]
Therefore
\begin{equation}
R_N(k)=\frac{1}{(2k+2N+1)^2}-
\sum_{\ell=1}^{\lfloor N/2\rfloor}
\frac{(-1)^{\ell+1}
\binom N\ell\binom{2N-2\ell}{N}}
{2\ell\binom{2N}{N}}
\left(
\frac{1}{2k+2N-2\ell+1}
-\frac{1}{2k+2N+1}
\right).
\nonumber
\end{equation}

Differentiating with respect to $k$, we obtain
\begin{align*}
R_N'(k)=
-\frac4{(2k+2N+1)^3}+
2\sum_{\ell=1}^{\lfloor N/2\rfloor}
\frac{(-1)^{\ell+1}
\binom N\ell\binom{2N-2\ell}{N}}
{2\ell\binom{2N}{N}}
\left(
\frac1{(2k+2N-2\ell+1)^2}
-\frac1{(2k+2N+1)^2}
\right).
\end{align*}
Then the right hand side of \eqref{eq:finite-WZ-derivative} is
\begin{align*}
4\sum_{k=0}^{\infty}
\left(
\frac1{(2k+1)^3}
-\frac1{(2k+2N+1)^3}
\right)+
2\sum_{\ell=1}^{\lfloor N/2\rfloor}
\frac{(-1)^{\ell+1}
\binom N\ell\binom{2N-2\ell}{N}}
{2\ell\binom{2N}{N}}
\sum_{k=0}^{\infty}
\left(
\frac1{(2k+2N-2\ell+1)^2}
-\frac1{(2k+2N+1)^2}
\right).
\end{align*}
For the first summation, the change of index $j=k+N$ gives
\[
\sum_{k=0}^{M}
\left(
\frac1{(2k+1)^3}
-\frac1{(2k+2N+1)^3}
\right)
=
\sum_{j=0}^{N-1}\frac1{(2j+1)^3}
-
\sum_{j=M+1}^{M+N}\frac1{(2j+1)^3}.
\]
Taking $M\to\infty$ yields
\[
\sum_{k=0}^{\infty}
\left(
\frac1{(2k+1)^3}
-\frac1{(2k+2N+1)^3}
\right)
=
\sum_{j=0}^{N-1}\frac1{(2j+1)^3},
\]
Similarly, for another summation, using $j=k+N-\ell$ and taking $M\to\infty$ yields
\[
\sum_{k=0}^{\infty}
\left(
\frac1{(2k+2N-2\ell+1)^2}
-\frac1{(2k+2N+1)^2}
\right)=
\sum_{j=N-\ell}^{N-1}\frac1{(2j+1)^2}.
\]

Eventually, \eqref{eq:finite-WZ-derivative} becomes
\[
2\sum_{n=0}^{N-1}
\frac{(3n+2)4^n}
{(2n+1)^3(n+1)\binom{2n}{n}^2}=
4\sum_{j=0}^{N-1}\frac{1}{(2j+1)^3}+2\sum_{\ell=1}^{\lfloor N/2\rfloor}
\frac{(-1)^{\ell+1}
\binom N\ell\binom{2N-2\ell}{N}}
{2\ell\binom{2N}{N}}\sum_{j=N-\ell}^{N-1}\frac{1}{(2j+1)^2}.
\]
This finishes the proof.
\end{proof}

\subsubsection{A differentiated Pfaff--Saalsch\"utz identity}
\begin{lem}\label{lem:pfaff}
For every integer $N\geq 0$,
\[
\sum_{\ell=1}^{N}
\frac{(-1)^\ell}{\ell}
\binom{N}{\ell}\binom{N+\ell}{\ell}
\sum_{j=1}^{\ell}\frac1{j^2}
=
-2\sum_{j=1}^{N}\frac1{j^3}.
\]
\end{lem}
\begin{proof}
Recall that the generalized hypergeometric function is defined by
\[
{}_pF_q\left(
\begin{matrix}
a_1,\ldots,a_p\\
b_1,\ldots,b_q
\end{matrix};z
\right)
=
\sum_{m=0}^{\infty}
\frac{(a_1)_m\cdots(a_p)_m}
     {(b_1)_m\cdots(b_q)_m}
\frac{z^m}{m!}.
\]
In the present case, one of the upper parameters is $-N$,
so the series terminates at $m=N$.

By the terminating Pfaff--Saalsch\"utz formula
\cite[Formula 16.4.3]{DLMF},
\begin{equation}\label{eq:pfaff-saalschutz}
    \begin{aligned}
      {}_3F_2\left(
\begin{matrix}
-N,\ N+1,\ u\\
1+z,\ 1+u-z
\end{matrix};1
\right):&=\sum_{\ell=0}^{N} \frac{(-N)_{\ell}(N+1)_{\ell}(u)_{\ell}}{(1+z)_{\ell}(1+u-z)_{\ell}\,\ell!}\\
&=\frac{(1-z)_N(1+z-u)_N}
     {(1+z)_N(1+u-z)_N}.  
    \end{aligned}
\end{equation}

Since the series on the left-hand side is finite, we may
differentiate term by term with respect to $u$:
\[
\left.
\frac{d}{du}
\frac{(u)_\ell}{(1+u-z)_\ell}
\right|_{u=0}
=
\frac{\frac{d}{du}|_{u=0}(u)_\ell}{(1+u-z)_\ell}.
\]
For $\ell\geq1$,
\[
\left.\frac{d}{du}(u)_\ell\right|_{u=0}
=(\ell-1)!,
\]
whereas $(u)_\ell|_{u=0}=0$. Thus the derivative of the left-hand side of
\eqref{eq:pfaff-saalschutz} at $u=0$ is
\[
\sum_{\ell=1}^{N}
\frac{(-N)_\ell(N+1)_\ell}
     {(1+z)_\ell(1-z)_\ell}
\frac{(\ell-1)!}{\ell!}=\sum_{\ell=1}^{N}\frac{(-N)_\ell(N+1)_\ell}
     {(1+z)_\ell(1-z)_\ell}\frac{1}{\ell}.
\]
Using
\[
(-N)_\ell(N+1)_\ell
=
(-1)^\ell(\ell!)^2
\binom{N}{\ell}\binom{N+\ell}{\ell}
\]
and
\[
(1+z)_\ell(1-z)_\ell
=
(\ell!)^2
\prod_{j=1}^{\ell}
\left(1-\frac{z^2}{j^2}\right),
\]
this becomes
\begin{equation}\label{eq:pfaff-left-derivative}
\sum_{\ell=1}^{N}
\frac{(-1)^\ell}{\ell}
\binom{N}{\ell}\binom{N+\ell}{\ell}
\prod_{j=1}^{\ell}
\left(1-\frac{z^2}{j^2}\right)^{-1}.
\end{equation}

We now differentiate the right-hand side of
\eqref{eq:pfaff-saalschutz}. Put
\[
R(u)=
\frac{(1-z)_N(1+z-u)_N}
     {(1+z)_N(1+u-z)_N}.
\]
Since $R(0)=1$, the logarithmic differentiation gives
\[
\left.\frac{d}{du}R(u)\right|_{u=0}
=
-\sum_{j=1}^{N}
\left(
\frac1{j+z}+\frac1{j-z}
\right)
=
-2\sum_{j=1}^{N}\frac{j}{j^2-z^2}.
\]
Consequently,
\begin{equation}\label{eq:pfaff-differentiated}
\sum_{\ell=1}^{N}
\frac{(-1)^\ell}{\ell}
\binom{N}{\ell}\binom{N+\ell}{\ell}
\prod_{j=1}^{\ell}
\left(1-\frac{z^2}{j^2}\right)^{-1}
=
-2\sum_{j=1}^{N}\frac{j}{j^2-z^2}.
\end{equation}

Finally, around $z=0$,
\[
\prod_{j=1}^{\ell}
\left(1-\frac{z^2}{j^2}\right)^{-1}
=
1+
z^2\sum_{j=1}^{\ell}\frac1{j^2}
+O(z^4),
\]
while
\[
\frac{j}{j^2-z^2}
=
\frac1j+\frac{z^2}{j^3}+O(z^4).
\]
Comparing the coefficients of $z^2$ in
\eqref{eq:pfaff-differentiated}, we obtain
\[
\sum_{\ell=1}^{N}
\frac{(-1)^\ell}{\ell}
\binom{N}{\ell}\binom{N+\ell}{\ell}
\sum_{j=1}^{\ell}\frac1{j^2}
=
-2\sum_{j=1}^{N}\frac1{j^3},
\]
as claimed.
\end{proof}

\subsubsection{Proof of the upper truncation congruence}
\
\newline
\indent We now apply the preceding Lemma to obtain the desired congruence for the upper truncation.

\textbf{Proof of Formula \eqref{eq:3log2 upper} in Theorem \ref{thm:3log2 congruence}:} We first record the following elementary reflection congruence.
For $1\leq k\leq (p-1)/2$,
\begin{equation}\label{eq:binom-reflection}
\frac1p\binom{2p-2k}{p-k}
\equiv
-\frac{2}{k\binom{2k}{k}}
\pmod p.
\end{equation}
Indeed, there is exactly one factor $p$ in $(2p-2k)!$
and none in $(p-k)!$. Hence
\[
\frac1p\binom{2p-2k}{p-k}
\equiv
-\frac{(p-2k)!}{((p-k)!)^2}
\pmod p.
\]
By Wilson's theorem,
\[
(p-r)!(r-1)!\equiv (-1)^r\pmod p,
\]
and therefore
\[
-\frac{(p-2k)!}{((p-k)!)^2}
\equiv
-\frac{((k-1)!)^2}{(2k-1)!}
=
-\frac{2}{k\binom{2k}{k}}
\pmod p.
\]\\

Applying \eqref{eq:binom-reflection} and making the
substitution $n=p-k$, we obtain
\begin{equation}
 \begin{aligned}
\frac1{p^2}
\sum_{n=\frac{p+1}{2}}^{p-1}
\frac{6n+1}{2n(2n+1)}
\binom{2n}{n}^2\frac1{2^{6n}}
&=\sum_{k=1}^{\frac{p-1}{2}} \frac{6(p-k)+1}{2(p-k)(2p-2k+1)} \left( \frac{1}{p} \binom{2p-2k}{p-k} \right)^2 \frac{1}{2^{6(p-k)}}\\
&\equiv
\sum_{k=1}^{\frac{p-1}{2}}
\frac{2(1-6k)2^{6(k-1)}}
{k^3(2k-1)\binom{2k}{k}^2}
\pmod p.
\label{eq:upper-reflection}
\end{aligned}    
\end{equation}
Here we have also used $64^{p-1}\equiv1\pmod p$. To use Lemma \ref{lem:finite-derived-wz}, we next reindex the sum by setting
\[
k=\frac{p-1}{2}-m,
\qquad
0\leq m\leq \frac{p-3}{2}.
\]
Then
\[
2(1-6k)
\equiv
4(3m+2)
\pmod p\quad\mathrm{and}\quad k^3(2k-1)
\equiv
\frac{(2m+1)^3(m+1)}{4}
\pmod p.
\]
Moreover, Fermat's little theorem gives
\[
2^{6(k-1)}=
2^{3(p-1)-6m-6}
\equiv
2^{-6m-6}
\pmod p.
\]
It remains to transform the central binomial coefficient.
For $0\leq r\leq (p-1)/2$, we have
\[
4^{-r}\binom{2r}{r}
\equiv
(-1)^r\binom{(p-1)/2}{r}
\pmod p.
\]
Applying this with $r=(p-1)/2-m$ and $r=m$, and using
the symmetry of the binomial coefficients, we obtain
\[
\binom{p-1-2m}{(p-1)/2-m}
\equiv
(-1)^{(p-1)/2}
16^{-m}\binom{2m}{m}
\pmod p,
\]
where $4^{(p-1)/2}\equiv1\pmod p$ has been used.
Consequently,
\[
\binom{p-1-2m}{(p-1)/2-m}^{-2}
\equiv
\frac{256^m}{\binom{2m}{m}^2}
\pmod p.
\]
Combining these congruences, we have
\begin{equation}\label{eq:upper-inverse-binomial}
\frac1{p^2}
\sum_{n=\frac{p+1}{2}}^{p-1}
\frac{6n+1}{2n(2n+1)}
\binom{2n}{n}^2\frac1{2^{6n}}
\equiv
\frac14
\sum_{m=0}^{\frac{p-3}{2}}
\frac{(3m+2)4^m}
{(2m+1)^3(m+1)\binom{2m}{m}^2}
\pmod p.
\end{equation}\\

We now apply Lemma \ref{lem:finite-derived-wz} with $N=\frac{p-1}{2}$. Since $\left\lfloor \frac{N}{2}\right\rfloor=\left\lfloor \frac{p}{4}\right\rfloor$,
Lemma \ref{lem:finite-derived-wz} gives
\begin{align}
\sum_{m=0}^{(p-3)/2}
\frac{(3m+2)4^m}
{(2m+1)^3(m+1)\binom{2m}{m}^2}=
2\sum_{j=0}^{(p-3)/2}\frac1{(2j+1)^3}
+
\sum_{\ell=1}^{\lfloor p/4\rfloor}
A_\ell
\sum_{j=(p-1)/2-\ell}^{(p-3)/2}
\frac1{(2j+1)^2},
\label{eq:lemma44-specialized}
\end{align}
where
\[
A_\ell=
\frac{(-1)^{\ell+1}
\binom{(p-1)/2}{\ell}
\binom{p-1-2\ell}{(p-1)/2}}
{2\ell\binom{p-1}{(p-1)/2}}.
\]

Using the standard relations between generalized binomial
coefficients and Pochhammer symbols (see
 \cite[Formula (1.2.6); Formula (5.2.8)]{DLMF}),
we have, for $1\leq\ell\leq\lfloor p/4\rfloor$,
\[
\binom{(p-1)/2}{\ell}
\equiv
(-1)^\ell4^{-\ell}\binom{2\ell}{\ell}
\pmod p,
\]
while
\[
\frac{\binom{p-1-2\ell}{(p-1)/2}}
{\binom{p-1}{(p-1)/2}}
=
\frac{\binom{(p-1)/2}{2\ell}}
{\binom{p-1}{2\ell}}
\equiv
16^{-\ell}\binom{4\ell}{2\ell}
\pmod p.
\]
Consequently,
\begin{align}
A_\ell\equiv
-\frac{1}{2\ell}
\frac{\binom{2\ell}{\ell}\binom{4\ell}{2\ell}}
{64^\ell}
=
-\frac{1}{2\ell}
\frac{(\frac14)_\ell(\frac34)_\ell}{(\ell!)^2} \equiv(-1)^{\ell+1}\frac{1}{2\ell}
\binom{\lfloor p/4\rfloor}{\ell}\binom{\lfloor p/4\rfloor+\ell}{\ell}
\pmod p.
\label{eq:Aell-congruence}
\end{align}
The last equivalence holds since $\{1/4,3/4\}\equiv\{-\lfloor p/4\rfloor, \lfloor p/4\rfloor+1\}\pmod p$.

Moreover, writing
\[
j=\frac{p-1}{2}-r,
\qquad 1\leq r\leq\ell,
\]
we have $2j+1=p-2r\equiv-2r\pmod p$. Hence
\begin{equation}\label{eq:terminal-odd-square}
\sum_{j=(p-1)/2-\ell}^{(p-3)/2}
\frac1{(2j+1)^2}
\equiv
\frac14\sum_{r=1}^{\ell}\frac1{r^2}
\pmod p.
\end{equation}

Substituting
\eqref{eq:Aell-congruence} and
\eqref{eq:terminal-odd-square}
into \eqref{eq:lemma44-specialized}, we obtain
\begin{align*}
\sum_{m=0}^{(p-3)/2}
\frac{(3m+2)4^m}
{(2m+1)^3(m+1)\binom{2m}{m}^2}\equiv
2\sum_{j=0}^{(p-3)/2}\frac1{(2j+1)^3}
-\frac18
\sum_{\ell=1}^{\lfloor p/4\rfloor}
\frac{(-1)^\ell}{\ell}
\binom{\lfloor p/4\rfloor}{\ell}\binom{\lfloor p/4\rfloor+\ell}{\ell}
\sum_{r=1}^{\ell}\frac1{r^2}
\pmod p.
\end{align*}
Applying Lemma \ref{lem:pfaff}, we conclude that
\begin{equation}\label{eq:upper-harmonic-reduction}
\sum_{m=0}^{(p-3)/2}
\frac{(3m+2)4^m}
{(2m+1)^3(m+1)\binom{2m}{m}^2}
\equiv
2\sum_{j=0}^{(p-3)/2}\frac1{(2j+1)^3}
+
\frac14\sum_{j=1}^{\lfloor p/4\rfloor}\frac1{j^3}
\pmod p.
\end{equation}\\

It remains to evaluate the two sums in \eqref{eq:upper-harmonic-reduction}.
Recall that the Bernoulli polynomials $B_n(x)$ are defined by
\[
\frac{te^{xt}}{e^t-1}
=
\sum_{n=0}^{\infty} B_n(x)\frac{t^n}{n!},\quad B_n=B_n(0);
\]
see \cite[Formula 24.2.3]{DLMF}.

We first note that the Bernoulli numbers which occur below are
$p$-integral. Indeed, by the von Staudt--Clausen theorem
(see \cite[Formula 24.10.1]{DLMF}), the denominator of
$B_{2r}$ is the product of the primes $q$ satisfying
$q-1\mid 2r$. Since $p-1\nmid p-3$, the denominator of
$B_{p-3}$ is prime to $p$.

By the special-value formula for Bernoulli polynomials
(see \cite[Formula 24.4.27]{DLMF}),
\[
B_n\left(\frac12\right)
=
\left(2^{1-n}-1\right)B_n.
\]
Taking $n=p-3$ and using Fermat's little theorem, we obtain
\begin{equation}\label{eq:Bernoulli-half}
B_{p-3}\left(\frac12\right)
=
\left(2^{4-p}-1\right)B_{p-3}
\equiv
7B_{p-3}
\pmod p.
\end{equation}
Moreover, by the reflection formula
(see \cite[Formula 24.4.3]{DLMF}),
\[
B_{p-3}\left(\frac14\right)
=
B_{p-3}\left(\frac34\right),
\]
since $p-3$ is even. The multiplication formula
\cite[Formula 24.4.21]{DLMF}, together with the
preceding symmetry, gives
\[
B_{p-3}\left(\frac14\right)
=
B_{p-3}\left(\frac34\right)
=
2^{3-p}B_{p-3}\left(\frac12\right).
\]
Hence, by \eqref{eq:Bernoulli-half} and Fermat's little theorem,
\begin{equation}\label{eq:Bernoulli-quarter}
B_{p-3}\left(\frac14\right)
=
B_{p-3}\left(\frac34\right)
\equiv
28B_{p-3}
\pmod p.
\end{equation}

We now use the power-sum formula
\cite[Formula 24.4.7]{DLMF},
\[
\sum_{j=1}^{M}j^r
=
\frac{B_{r+1}(M+1)-B_{r+1}}{r+1}.
\]
Since $j^{-3}\equiv j^{p-4}\pmod p$ for $1\leq j\leq p-1$, taking $r=p-4$ gives
\begin{equation}\label{eq:power-sum-cubic}
\sum_{j=1}^{M}\frac1{j^3}
\equiv
\frac{B_{p-3}(M+1)-B_{p-3}}{p-3}
\pmod p.
\end{equation}
Taking $M=(p-1)/2$ in \eqref{eq:power-sum-cubic}, and using
$(p+1)/2\equiv 1/2\pmod p$ together with
\eqref{eq:Bernoulli-half}, we obtain
\begin{equation}\label{eq:half-harmonic-cubic}
\begin{aligned}
\sum_{j=1}^{(p-1)/2}\frac1{j^3}\equiv
\frac{B_{p-3}(\frac12)-B_{p-3}}{p-3} \equiv
-2B_{p-3}
\pmod p.
\end{aligned}
\end{equation}
Similarly, let $M=\lfloor p/4\rfloor$, since $\{1/4,3/4\}\equiv\{-\lfloor p/4\rfloor, \lfloor p/4\rfloor+1\}\pmod p$, by \eqref{eq:Bernoulli-quarter} and
\eqref{eq:power-sum-cubic},
\begin{equation}\label{eq:quarter-harmonic-cubic}
\begin{aligned}
\sum_{j=1}^{\lfloor p/4\rfloor}\frac1{j^3}\equiv
\frac{28B_{p-3}-B_{p-3}}{p-3}  \equiv
-9B_{p-3}
\pmod p.
\end{aligned}
\end{equation}

It remains to evaluate the sum over the odd denominators. Observing that $2j+1$ represents the odd numbers from 1 to $p-2$, we can write the odd residues modulo $p$ in the form
$p-2r$, $1\leq r\leq(p-1)/2$. Then
\begin{align}
\sum_{j=0}^{(p-3)/2}\frac1{(2j+1)^3}\equiv
\sum_{r=1}^{(p-1)/2}\frac1{(p-2r)^3}\equiv
-\frac18
\sum_{r=1}^{(p-1)/2}\frac1{r^3}\equiv
\frac14 B_{p-3}
\pmod p,
\label{eq:odd-harmonic-cubic}
\end{align}
where the last congruence follows from
\eqref{eq:half-harmonic-cubic}. Substituting \eqref{eq:quarter-harmonic-cubic} and
\eqref{eq:odd-harmonic-cubic} into \eqref{eq:upper-harmonic-reduction}, we have
\begin{align}
\sum_{m=0}^{(p-3)/2}
\frac{(3m+2)4^m}
{(2m+1)^3(m+1)\binom{2m}{m}^2}\equiv
-\frac74 B_{p-3}
\pmod p.
\label{eq:inverse-binomial-Bernoulli}
\end{align}

Finally, combining \eqref{eq:inverse-binomial-Bernoulli}
with \eqref{eq:upper-inverse-binomial}, 
\[
\frac1{p^2}
\sum_{n=(p+1)/2}^{p-1}
\frac{6n+1}{2n(2n+1)}
\binom{2n}{n}^2\frac1{2^{6n}}
\equiv
-\frac7{16}B_{p-3}
\pmod p.
\]
i.e.,
\[
\sum_{n=(p+1)/2}^{p-1}
\frac{6n+1}{2n(2n+1)}
\binom{2n}{n}^2\frac1{2^{6n}}
\equiv
-\frac7{16}p^2B_{p-3}
\pmod{p^3}.
\]

This proves Formula \eqref{eq:3log2 upper} and completes the proof of Theorem \ref{thm:3log2 congruence}. \hfill$\square$